\documentclass[12pt,oneside,draft]{amsart}

      \usepackage{amssymb}
      \usepackage[english]{babel}
      \usepackage{graphicx}
      \usepackage[applemac]{inputenc}
      \usepackage{float}
      \usepackage{wrapfig}
      \usepackage[a4paper=true,%
      breaklinks=true,%
      colorlinks=true,%
      linkcolor={blue},%
      citecolor={red},%
      pdftitle={A note on the equation $\zeta(s)=1$},%
      pdfauthor={Richard P. Brent, Jan van de Lune},%
      pdfkeywords={snapshot},%
      pdfstartview=FitH,
       ]{hyperref}

      \theoremstyle{plain}
      \newtheorem{theorem}{Theorem} 
      \newtheorem{lemma}[theorem]{Lemma}
      \newtheorem{corollary}[theorem]{Corollary}
      
      \theoremstyle{definition}

      \theoremstyle{remark}

      \newcommand{\N}{{\mathbb N}}
      \newcommand{\Z}{{\mathbb Z}}

      \newfont{\cmbsy}{cmbsy10}
      \newfont{\cmmib}{cmmib10}

      \def\Re{\mathrm{Re\,}}

      \makeatletter
      \def\@setcopyright{}
      \def\serieslogo@{}
      \makeatother

\begin{document}

   
   \author{Richard P. Brent$\,\,$}

   \address{Mathematical Sciences Institute, 
   Australian National University, Canberra, ACT 0200, Australia}
   \email{polya@rpbrent.com}


   \author{$\,\,$Jan van de Lune }
   \address{\noindent Langebuorren 49, 9074 CH Hallum, The Netherlands 
   \newline(Formerly at CWI, Amsterdam ) }
   \email{j.vandelune@hccnet.nl}
   

   \title[P\'olya's observation concerning Liouville's function]
   {A note on P\'olya's observation concerning Liouville's function}


   \begin{abstract}
   We show that a certain weighted mean of the Liouville function
   $\lambda(n)$ is negative. In this sense, we can say that the Liouville
   function is negative ``on average''.
   \end{abstract}




   \dedicatory{Dedicated to Herman J.~J.~te Riele on the occasion
     of his retirement from the CWI in January 2012}

   \date{\today}

   \maketitle
\section{Introduction}

   For $n\in \N$ let $n=\prod_{p|n} p^{e_p(n)}$ be the canonical prime 
   factorization of $n$ and let $\Omega(n):=\sum_{p\mid n} e_p(n)$.
   Here (as always in this paper) $p$ is prime.
   Thus, $\Omega(n)$ is the total number of prime factors of  $n$,  
   counting multiplicities.  For example: $\Omega(1)=0$, $\Omega(2)=1$, 
   $\Omega(4)=2$, $\Omega(6)=2$, $\Omega(8)=3$, $\Omega(16)=4$, 
   $\Omega(60)=4$, etc. 
   
   Define  Liouville's multiplicative function  $\lambda(n)=(-1)^{\Omega(n)}$. 
   For example $\lambda(1)=1$, $\lambda(2)=-1$, $\lambda(4)=1$, etc. 
   The M\"obius function $\mu(n)$ may be defined
   to be $\lambda(n)$ if $n$ is square-free, and $0$ otherwise.
   
   It is well-known, and follows easily from the Euler product for the
   Riemann zeta-function $\zeta(s)$, that $\lambda(n)$ has the Dirichlet
   generating function
   \begin{displaymath}
   \sum_{n=1}^\infty \frac{\lambda(n)}{n^s} = \frac{\zeta(2s)}{\zeta{(s)}}
   \end{displaymath}
   for $\Re(s) > 1$. This provides an alternative definition of $\lambda(n)$.

   Let $L(n) := \sum_{k \le n}\lambda(k)$ be the summatory function of
   the Liouville function; similarly $M(n) := \sum_{k \le n}\mu(k)$
   for the M\"obius function.

   The topic of this note is closely related to P\'olya's
   conjecture~\cite[1919]{Polya} that $L(n) \le 0$ for $n\ge2$. 

   P\'olya verified this for  $n \le 1500$  and Lehmer \cite[1956]{Lehmer} 
   checked it 
   for  $n \le 600\,000$. However, Ingham \cite[1942]{Ingham} cast doubt on
   the plausibility of P\'olya's conjecture by showing that it would imply
   not only the Riemann Hypothesis and simplicity of the zeros of
   $\zeta(s)$, but also the linear dependence over the rationals
   of the imaginary parts
   of the zeros $\rho$ of $\zeta(s)$ in the upper half-plane.
   Ingham cast similar doubt
   on the Mertens conjecture $|M(n)| \le \sqrt{n}$, 
   which was subsequently disproved in a
   remarkable \emph{tour de force} by Odlyzko and te Riele \cite[1985]{OR}.
   More recent results and improved bounds were given by
   Kotnik and te Riele \cite[2006]{KR}; see also 
   Kotnik and van de Lune \cite[2004]{KL}.

   In view of Ingham's results, it was no surprise when
   Haselgrove showed \cite[1958]{Haselgrove} that P\'olya's conjecture
   is false. He did not give an explicit counter-example,
   but his proof suggested that $L(u)$ might be positive in
   the vicinity of $u \approx 1.8474 \times 10^{361}$. 

   Sherman Lehman \cite[1960]{Lehman} gave an algorithm for calculating $L(n)$
   similar to Meissel's \cite[1885]{Meissel} formula 
   for the prime-counting function $\pi(x)$,
   and found the counter-example $L(906\,180\,359) = +1$.
   
   Tanaka \cite[1980]{Tanaka} found the smallest counter-example  
   $L(n) = +1$ for $n = 906\,150\,257$. 
   Walter M.~Lioen and Jan van de Lune [\emph{circa} 1994] scanned the range
   $n \le 2.5 \times 10^{11}$ using a fast sieve, but found no
   counter-examples beyond those of Tanaka.  More recently, Borwein,
   Ferguson and Mossinghoff~\cite[2008]{BFM} showed that
   $L(n) = +1\,160\,327$ for $n = 351\,753\,358\,289\,465$.

   Humphries \cite{Humphries1,Humphries2} showed that, under certain
   plausible but unproved hypotheses (including the Riemann Hypothesis),
   there is a limiting logarithmic distribution of $L(n)/\sqrt{n}$, and
   numerical computations show that 
   the logarithmic density of the set $\{n\in\N | L(n) < 0\}$
   is approximately $0.99988$. 
   Humphries' approach followed that of Rubinstein and Sarnak~\cite{RS},
   who investigated ``Chebyshev's bias'' in prime ``races''.

   Here we show in an elementary manner, and without any unproved
   hypotheses, that $\lambda(n)$ is (in a certain sense) ``negative
   on average''. To prove this, all that we need are some well-known facts
   about Mellin transforms, and the functional equation for the Jacobi theta
   function (which may be proved using Poisson summation).  Our main result
   is:

   \begin{theorem} \label{thm:1}
   There exists a positive constant $c$ such that for every
   (fixed) $N\in\N$
   \begin{displaymath}
   \sum_{n=1}^\infty \frac{\lambda(n)}{e^{n\pi x}+1}=-\frac{c}{\sqrt{x}}
   +\frac12+O(x^N) \qquad \text{as}\quad x\downarrow 0.
   \end{displaymath}
   \end{theorem}

Thus, a weighted mean of $\{\lambda(n)\}$, with positive weights initially
close to a constant $(1/2)$ and becoming small for $n \gg 1/x$, is negative 
for $x < x_0$ and tends to $-\infty$ as $x \downarrow 0$.

In the final section we mention some easy results on the M\"obius function
$\mu(n)$ to contrast its behaviour with that of $\lambda(n)$.

\pagebreak

\section{Proof of Theorem \ref{thm:1}}

We prove Theorem \ref{thm:1} in three steps.

   \bigskip
   
   \emph{Step 1.} 
   For $x > 0$, 
   \begin{displaymath}
   \sum_{n=1}^\infty \frac{\lambda(n)}{e^{n\pi x}-1}=\phi(x)
   = \frac{\theta(x) -1}{2}\,,
   \end{displaymath}
   where
   \begin{displaymath}
   \phi(x):=\sum_{k=1}^\infty e^{-k^2\pi x}, \qquad
   \theta(x):= \sum_{k \in \Z} e^{-k^2\pi x}.
   \end{displaymath}
   
   \emph{Step 2.}
   For $x > 0$,
   \begin{displaymath}
   \sum_{n=1}^\infty \frac{\lambda(n)}{e^{n\pi x}+1}=\phi(x)-2\phi(2x).
   \end{displaymath}
   
   \emph{Step 3.} Theorem \ref{thm:1} now follows from the functional
   equation 
   \begin{displaymath}
   \theta(x) = \frac{1}{\sqrt{x}}\,\theta\left(\frac{1}{x}\right)
   \end{displaymath}
   for the Jacobi theta function $\theta(x)$.
   
   \begin{proof}[Proof of Theorem \ref{thm:1}]~\\

   (1) In the following, we assume that $\Re(s) > 1$, so the Dirichlet
   series and integrals are absolutely convergent, and interchanging the
   orders of summation and integration is easy to justify.

   As mentioned above, it is well-known that
   \begin{displaymath}
   \sum_{n=1}^\infty\frac{\lambda(n)}{n^s}
   =\prod_p\left({1+p^{-s}}\right)^{-1}
   =\prod_p\frac{1-p^{-s}}
   {1-p^{-2s}} = \frac{\zeta(2s)}{\zeta(s)}.
   \end{displaymath}
   Define
   \begin{displaymath}
   f(x):=\sum_{n=1}^\infty \frac{\lambda(n)}{e^{n x}-1},\qquad (x>0).
   \end{displaymath}
   We will use the well known fact that if two sufficiently well-behaved 
   functions (such as ours below) have the same Mellin transform 
   then the functions are equal.  

   The Mellin transform of  $f(x)$ is
   \begin{eqnarray*}
   {F}(s) &:=& 
   \int_0^\infty f(x)x^{s-1}\,{\rm d}x=\int_0^\infty 
   \sum_{n=1}^\infty \frac{\lambda(n)}{e^{nx}-1}x^{s-1}\,{\rm d}x\\
   &=&\sum_{n=1}^\infty \lambda(n)\int_0^\infty\frac{x^{s-1}}{e^{nx}-1}    
   \,{\rm d}x
   =\Bigl(\sum_{n=1}^\infty\frac{\lambda(n)}{n^s}\Bigr)\times    
   \int_0^\infty\frac{x^{s-1}}{e^{x}-1}\,{\rm d}x\\
   &=&\frac{\zeta(2s)}{\zeta(s)}\times
   \zeta(s)\Gamma(s) =\zeta(2s)\Gamma(s).
   \end{eqnarray*}

   We also have
   \begin{eqnarray*}
   \int_0^\infty \phi\Bigl(\frac{x}{\pi}\Bigr)x^{s-1}\,{\rm d}x
   &=&\int_0^\infty
   \Bigl(\sum_{n=1}^\infty e^{-n^2x}\Bigr)x^{s-1}\,{\rm d}x\\
   &=&\Bigl(\sum_{n=1}^\infty
   \frac{1}{n^{2s}}\Bigr)\times\int_0^\infty e^{-x}x^{s-1}\,{\rm d}x=
    \zeta(2s)\Gamma(s),
   \end{eqnarray*}
   so the Mellin transforms of $f(x)$ and of $\phi(x/\pi)$ are
   identical.  Thus
   $f(x)=\phi({x}/{\pi})$.
   Replacing $x$ by $\pi x$, we see that
   \begin{displaymath}
   \sum_{n=1}^\infty \frac{\lambda(n)}{e^{n\pi x}-1}=\sum_{k=1}^\infty 
   e^{-k^2\pi x},
   \end{displaymath}
   completing the proof of step (1).
   \smallskip
   
   (2) Observe that 
   \begin{displaymath}
   \frac{1}{e^{n\pi x}+1}=\frac{1}{e^{n\pi x}-1}-
   \frac{2}{e^{2n\pi x}-1},
   \end{displaymath} 
   so, from step (1),
   \begin{displaymath}
   \sum_{n=1}^\infty \frac{\lambda(n)}
   {e^{n\pi x}+1}=\phi(x)-2\phi(2x).
   \end{displaymath}
   \smallskip   

   (3) Using the functional equation for $\theta(x)$, we easily find
   that 
   \begin{displaymath}
   \phi(x)-2\phi(2x)=-\frac{c}{\sqrt{x}}+\frac12+
   \frac{1}{\sqrt{x}}\Bigl(\phi\Bigl(\frac{1}{x}\Bigr)-\sqrt{2}\,
   \phi\Bigl(\frac{1}{2x}\Bigr)\Bigr)
   \end{displaymath}
   with
   $c={(\sqrt{2}-1)}/{2} > 0$,
   proving our claim, since the ``error'' term is bounded by
   $\phi(1/x)/\sqrt{x} \sim \exp(-\pi/x)/\sqrt{x} = O(x^N)$
   as $x \downarrow 0$ (for any fixed exponent $N$).
   \end{proof}
                   
\section{Remarks on the M\"obius function}
                                            
We give some further applications of the identity
\begin{equation}
\tag{$*$}\label{eq:starred}
\frac{1}{z+1}=\frac{1}{z-1}-\frac{2}{z^2-1}
\end{equation}
that we used (with $z = e^{n\pi x}$) in proving step (2) above.
   
\begin{lemma} \label{lemma:2}
For $|x| < 1$, we have
\begin{displaymath}
   {\sum_{n=1}^\infty \mu(n)\frac{x^n}{x^n+1}} = x - 2x^2.
\end{displaymath}
\end{lemma}  

\begin{proof}[Proof]
Assume that $|x| < 1$.
   It is well known that 
   \begin{displaymath}
   {\sum_{n=1}^\infty \mu(n)\frac{x^n}{1-x^n} = x},
   \end{displaymath}
in fact this ``Lambert series'' identity is equivalent to the Dirichlet
series identity $\sum \mu(n)/n^s = 1/\zeta(s)$.
Writing $y = 1/x$, we have
   \begin{displaymath}
   \sum_{n=1}^\infty \frac{\mu(n)}{y^n-1} = 1/y.
   \end{displaymath}
If follows on taking $z = y^n$ in our identity (\ref{eq:starred}) that
   \begin{equation*}
   \sum_{n=1}^\infty \frac{\mu(n)}{y^n+1}=
   \sum_{n=1}^\infty \frac{\mu(n)}{y^n-1}-
   2\sum_{n=1}^\infty \frac{\mu(n)}{y^{2n}-1}
   = y^{-1} - 2y^{-2}.
   \end{equation*}
Replacing $y$ by $1/x$ gives the result.
   \end{proof}

\begin{corollary}
\begin{displaymath}
   {\sum_{n=1}^\infty \frac{\mu(n)}{2^n+1}} = 0.
\end{displaymath}
\end{corollary}  

\begin{proof}
Take $x = 1/2$ in Lemma \ref{lemma:2}.
\end{proof}

   
   If follows from Lemma~\ref{lemma:2} that
   \begin{displaymath}
   \lim_{x\uparrow1}
   \sum_{n=1}^\infty \mu(n)\frac{x^n}{x^n+1} =-1,
   \end{displaymath}
   so that one might say that in this sense $\mu(n)$ is 
   negative on average.
   However, this is much weaker than what we showed in Theorem~\ref{thm:1} 
   for $L(n)$,
   where the corresponding sum tends to $-\infty$.
   The ``complex-analytic'' reason for this difference
   is that $\zeta(2s)/\zeta(s)$ has a pole (with negative residue) 
   at $s = 1/2$, but $1/\zeta(s)$ is regular at $s=1$.

\end{document}